\documentclass[a4paper,12pt]{amsart}
\usepackage{amsmath}
\usepackage[bookmarksnumbered, colorlinks, plainpages]{hyperref}
\usepackage{amsfonts}
\usepackage{setspace}
\usepackage{amsthm}
\usepackage{amssymb}
\usepackage{graphicx}
\usepackage{stmaryrd}
\usepackage{yfonts}
\usepackage{color}
\usepackage{cite}
\usepackage[all,cmtip]{xy}
\usepackage{amscd,graphicx, color}
\date{}

 \theoremstyle{plain}
\newtheorem{theo}{Theorem}[section] \theoremstyle{plain}
 \theoremstyle{plain}
\newtheorem{coro}[theo]{Corollary} \theoremstyle{definition}
\newtheorem{defn}[theo]{Definition} \theoremstyle{definition}
 \theoremstyle{remark}
\newtheorem{rem}[theo]{Remark} \theoremstyle{remark}

\newtheorem{proposition}[theo]{Proposition}

\begin{document}
\setcounter{page}{1}

\title{ A Note on Submaximal Operator Space Structures }

\author{Vinod Kumar. P }
\address{Department of Mathematics,\\ Thunchan Memorial Govt College, Tirur,\\
Kerala , India. \em E-mail: vinodunical@gmail.com  \em}
   \author{  M. S. Balasubramani}
   \address{Department of Mathematics, University of Calicut,
Calicut University. P. O., Kerala, India. \em E-mail:
 msbalaa@rediffmail.com \em}

\date{}
\maketitle

\begin{abstract}

 In this note, we consider the smallest submaximal space
structure $\mu(X)$ on a Banach space $X$.  We derive a
characterization of $\mu(X)$ up to complete isometric isomorphism
in terms of a universal property. Also, we show that an injective
Banach space has a unique submaximal space structure and we
explore some duality relations of $\mu$-spaces.
\flushleft {AMS Mathematics Subject Classification(2000) No: 46L07,47L25}

\end{abstract}
\thispagestyle{empty}
\section{Introduction and Preliminaries}
 \noindent An
operator space consists of a Banach space $X$ and an isometric
embedding $J: X \rightarrow \mathcal {B(\mathcal H)}$, for some
Hilbert space $\mathcal H $. In contrast to the Banach space case,
an operator space carries not just a complete norm on $X$, but
also a sequence of complete norms on $M_n(X)$, the space of $n
\times n$ matrices on $X$, for every $n\in \mathrm{N}$. These
matrix norms are obtained via the natural identification of
$M_n(X)$ as a subspace of $M_n(\mathcal {B(\mathcal H)}) \approx
\mathcal {B(\mathcal H}^n)$, where  $\mathcal H^n $ is the Hilbert
space direct sum of $n$ copies of $\mathcal H $. In 1988, Z-J.
Ruan \cite{r88} characterized  the sequence of matrix norms on a
Banach space $X$ that makes $X$ an operator space, in terms of two
properties  of matrix norms, known as Ruan's axioms.
 An (\em abstract) operator space \em is a pair $ (X, \{  \left\| . \right\| _n \}_{n \in N} )$
 consisting of a linear space $X$ and a complete norm $ \left\| . \right\| _n $ on  $ M_n(X) $ for every $ n\in N $,
 such that
$(R1)$ $\left\| \alpha x \beta \right\|_n \leq \left\| \alpha
\right\| \left\| x \right\|_n \left\| \beta \right\|$ for all
$\alpha, \beta \in M_n $ and for all $x \in M_n(X) $, and $(R2)$ $
\left\| x \oplus y \right\|_{m+n} = max \{  \left\| x \right\|_m ,
\left\| y \right\|_n \} $ for all $x \in M_m(X) $, and for all $y
\in M_n(X) $, where $x \oplus y $ denotes the matrix \( \left[
\begin{array}{cc}
x & 0 \\
0 & y
\end{array} \right] \) in $ M_{m+n}(X) $ where  $0$ stands for zero matrices of appropriate
orders.  The sequence of matrix norms $ \{  \left\| . \right\| _n
\}_{n \in \mathrm{N}} $ is called an \em operator space structure
\em
 on the linear space $X$.  An opeartor space structure on a normed space $(X, \left\|.\right\|)$ will usually mean a sequence of matrix norms
 $ \{  \left\| . \right\| _n \}_{n \in \mathrm{N}} $ as above, but with $   \left\| . \right\|_1 = \left\|.\right\| $
and in  that case, we say $ \{  \left\| . \right\| _n \}_{n \in
\mathrm{N}} $ is an admissible operator space structure on $X$. If
$X$ is an abstract operator space, then there exists  a linear
complete isometry $\varphi
  : X \rightarrow \mathcal B(\mathcal H)$ for some Hilbert space
  $\mathcal H $. Also, the matrix norms on
  an operator space induced via the embedding $ M_n(X) \subset M_n(\mathcal B(\mathcal H))
  $ satisfies the Ruan's axioms. Thus Ruan's characterization allows us to describe an operator space
  in an abstract way free of any concrete representation on a Hilbert space.\\
  If $X$ is a Banach space, the  closed unit
ball $\{ x \in X ; \left\| x  \right\| \leq 1 \}$ is denoted by
$Ball($X$)$. If  $X$ and $Y$ are  operator spaces and $\varphi : X
\rightarrow Y$ is a linear map,  $ \varphi^ {(n) } : M_n(X)
\rightarrow M_n(Y)$,
 given  by $[x_{ij}] \rightarrow [\varphi (x_{ij})]$, with $[x_{ij}] \in M_n(X)$ and  $n \in \mathrm{N}$ , determines a  linear map from $M_n(X)$ to $M_n(Y)$.
The \em complete bound norm \em (in short \em cb-norm \em) of
$\varphi$ is defined as  $ \left\| \varphi \right\|_{cb}  =
\mbox{sup}\{ \left\| \varphi^{(n)} \right\| ;  n \in \mathrm{N} \}
$.  $ \varphi $ is \em completely bounded \em if $ \left\| \varphi
\right\|_{cb} <  \infty $.
 $ \varphi $ is  a  \em complete isometry \em if
each map $ \varphi^ {(n) } : M_n(X) \rightarrow M_n(Y)$
 is an isometry.
 If $ \varphi $ is a complete isometry, then  $ \left\| \varphi \right\|_{cb}  = 1 $.
  If $ \left\| \varphi \right\|_{cb}  \leq
 1$ , $\varphi$ is said to be a complete contraction.
 If $\varphi : X \rightarrow Y $ is a completely bounded linear bijection and if its inverse is also completely bounded,
 then $ \varphi $ is said to be a \em complete isomorphism \em .
 Two  operator spaces are considered to be the same if there is
a complete isometric isomorphism  from $X$ to $Y$.  In that case,
we write $ X \approx Y$  \em completely isometrically \em .
\\
A Banach space $Z$ is \em injective \em  if for any Banach spaces
$X$ and $Y$  where  $Y$ contains $X$ as a closed subspace, and for
any completely bounded linear map $\varphi: X \rightarrow Z$,
there exists a  bounded linear extension $\tilde{\varphi}: Y
\rightarrow Z $ such that $\tilde{\varphi}|_{X} = \varphi$ and  $
\left\|\tilde{ \varphi }\right\| =  \left\| \varphi\right\|$. In a
similar manner, an operator space $Z$ is \em injective \em
\cite{ce77} if for any operator spaces $X$ and $Y$  where  $Y$
contains $X$ as a closed subspace, and for any completely bounded
linear map $\varphi: X \rightarrow Z$, there exists a completely
bounded extension $\tilde{\varphi}: Y \rightarrow Z $ such that
$\tilde{\varphi}|_{X} = \varphi$ and  $ \left\|\tilde{ \varphi
}\right\|_{cb} =  \left\| \varphi\right\|_{cb}$. An operator space
$X$ is \em homogeneous \em \cite{gp96} if each bounded linear
operator
 $\varphi$ on $X$ is completely bounded with $\left\| \varphi \right\|_{cb} =
   \left\| \varphi \right\|$. More
information about operator spaces and completely bounded mappings
may be found in the papers \cite{er91}, \cite{er93}, \cite{vp92}
and \cite{sm83} or in the books \cite{er02}, \cite{pa02} and
\cite{gp03}.
\section{Minimal and Maximal operator space structures }
A given Banach space has in general many realizations as an
operator space. A very basic question in operator space theory is
to exhibit some particular operator space structures on a given
Banach space $X$.
 In the most general situation, Blecher and  Paulsen \cite{bp91} achieved this by noting
 that the set of all operator space structures admissible on a given Banach
space $X$ admits a minimal and maximal element.  These structures
were further investigated in \cite{vp92} and \cite{vp96}. By Hahn-
Banach theorem, it follows that any subspace of a minimal operator
space is again minimal. Quotients of  minimal operator spaces are
called  $Q $-spaces \cite{er}, and they need not be minimal. Also,
the category of $Q $-spaces is stable under taking quotients and
subspaces.
 An operator space $X$ is said to be \em submaximal \em\cite{gp03} if it embeds
completely isometrically into a maximal operator space Y.
Generally, a submaximal space need not be maximal, but maximality
passes to quotients\cite{gp03}. Subspace structure of various
maximal operator spaces were further studied in \cite{tim04}.
  \\
 If $ X $ is a Banach space, then there is a  minimal
 operator space structure on $X$,  denoted by  $Min(X)$, and this quantization
   is characterized by
   the property that for any arbitrary operator space $Y$ and for
   any bounded linear map $\varphi : Y \rightarrow Min(X)$ is
   completely bounded and satisfies  $\left\| \varphi : Y \rightarrow Min(X)  \right\|_{cb} =
    \left\| \varphi : Y \rightarrow X  \right\|$.
An operator space $X$ is minimal if $Min(X)= X$. Also,  an
operator space is minimal if and only if it is completely
isometric to a subspace of a commutative C*- algebra.
  If $X$ is a Banach space, there is a maximal way to
consider it as an operator space. The matrix norms given by
$\left\| [x_{ij}]\right\|_n = \mbox{sup}\{  \left\|
[\varphi(x_{ij})] ) \right\| ;  \varphi \in Ball (B(X, Y))\}$
where the supremum is taken over all operator spaces $Y$ and all
linear maps $\varphi \in Ball (B(X, Y))$, makes $X$ an operator
space. This operator space is denoted by  $Max(X)$ and is called
the maximal operator space structure on $X$. For $[x_{ij}] \in
M_n(X)$, we write $ \left\| [x_{ij}]
   \right\|_{Max(X)}$ to denote its norm as an element of $ M_n(Max(X))$.
An operator space $X$ is maximal if $Max(X) = X $. By Ruan's
theorem, we also have $\left\| [x_{ij}]\right\|_{Max(X)} =
\mbox{sup}\{  \left\| [\varphi(x_{ij})] ) \right\| ;  \varphi \in
Ball(B(X, B(\mathcal{H})))\}$ where the supremum is taken over all
Hilbert spaces $\mathcal{H}$ and all linear maps $\varphi \in
Ball(B(X, B(\mathcal{H})))$. By the definition of $Max(X)$, any
operator space structure  on $X$ is  smaller than $Max(X)$. This
maximal quantization of a normed space is characterized by
   the property that for any arbitrary operator space $Y$,
   any bounded linear map $\varphi : Max(X) \rightarrow Y$ is
   completely bounded and satisfies  $\left\| \varphi : Max(X) \rightarrow Y  \right\|_{cb} =
    \left\| \varphi : X \rightarrow Y  \right\|$.
    If $X$ is any operator space, then the identity map on $X$
    defines completely contractive maps $Max(X)\rightarrow X \rightarrow
    Min(X)$.
 For any Banach space $X$, we have the following duality relations \cite{b92}: $Min(X)^* \approx Max(X^*) $ and
 $ Max(X)^* \approx Min (X^*)$ completely isometrically.
 \\
 \hspace{0.5in} Just like every operator space embeds completely
isometrically into $B(\mathcal{H})$ for some Hilbert space
$\mathcal{H}$, every submaximal space embeds completely
isometrically into $Max(B(\mathcal{H}))$ for some Hilbert space
$\mathcal{H}$. To see this, let $X\subset Y$, where $Y$ is a
maximal operator space. Also, let $\iota : X \to B(\mathcal{H})$
be a complete isometric inclusion.
  Since $B(\mathcal{H})$ is injective, the
inclusion $\iota : X \to B(\mathcal{H})$ extends to a complete
contraction $\varphi: Y\rightarrow B(\mathcal{H}) $. Since $Y$ is
maximal, $ \left\| \varphi : Y \rightarrow
Max(B(\mathcal{H}))\right\|_{cb} = \left\| \varphi : Y \rightarrow
Max(B(\mathcal{H}))\right\|\leq \left\| \varphi : Y \rightarrow
B(\mathcal{H})\right\|_{cb}\leq 1 $. Let $\widetilde{\iota}=
\varphi|_X$, then $ \left\| \widetilde{\iota} : X \rightarrow
Max(B(\mathcal{H}))\right\|_{cb} \leq 1 $. If
$\widetilde{\iota}(X)= \widetilde{X} $, by the definition of
maximal operator spaces, we have $ \left\| \widetilde{\iota}^{-1}
: \widetilde{X} \rightarrow X \right\|_{cb} \leq 1 $. Thus
$\widetilde{\iota} : X \to \widetilde{X}$
 is a completely isometric isomorphism.
\\

In the following, we consider the smallest submaximal space
structure on a Banach space $X$, namely the $\mu$-space structure
which is denoted by $\mu(X)$. We prove that $\mu(X)$ will be
homogeneous. We also derive a universal property of $\mu$-spaces
which distinguishes it among other submaximal spaces. By making
use of this property, we show that the class of $\mu$-spaces is
stable under taking subspaces. Finally, we explore the duality
relations of $\mu$-spaces.
\section{Main results}
Just like minimal and maximal operator space structures, we
 have a \em{ minimal}  and a \em{maximal } way to view a Banach space $X$ as
 a submaximal space, which we denote as $Min_S(X)$ and $Max_S(X)$
 respectively.
 From the definition of a submaximal space it follows that
 $Max_S(X)=Max(X)$. T. Oikhberg \cite{tim04} introduced the  $\mu$-space structure
 on a Banach space $X$ and  proved that $Min_S(X)= \mu(X)$.
Suppose $X$ is a Banach space.  Note  that $Min(X)$ is the
operator space structure on $X$ inherited by regarding $X \subset
C(K)$, where $ K = Ball(X^*)$, the  closed unit ball of the dual
space of $X$ with its weak* topology. Also, from \cite{b92},  $
Max(X)^* = Min (X^*)$, so that
 $ Max(X)^{**} = (Min (X^*))^* = Max(X^{**})$ completely isometrically.  Since
$X\hookrightarrow X^{**}$, we have $ Max(X)\hookrightarrow
Max(X^{**})$ completely isometrically.
\begin{defn}
An operator space $X$ is a $\mu$-space if it embeds completely
isometrically into $Max(C(K)^{**})$,  where $K = Ball(X^*)$  the
unit closed ball of the dual space of $X$ with its weak* topology.
\end{defn}
 A Banach space $X$, with the above defined $\mu$-space structure is  denoted  by $\mu(X)$ and the corresponding sequence of matrix norms
 by  $ \{  \left\| . \right\|^{\mu} _{n} \}_{n \in \mathrm{N}} $.
Note that the $\mu$-space structure on a Banach space $X$ is an
admissible operator space structure on $X$.
\begin{rem} Suppose  that  $X$ and $Y$ are injective Banach spaces and
 $E$ and $F$ are isomorphic (isometric) closed subspaces of $X$
and $Y$ respectively. Let $\varphi: E \rightarrow F$ be an
isomorphism. Since $Y$ is injective, there exists a map
$\tilde{\varphi}: Max(X) \rightarrow Max(Y)$ such that
$\tilde{\varphi}|_{E} =
 \varphi $ and $\left\|\tilde{ \varphi }\right\|=  \left\|
\varphi\right\|$. Since $X$ has maximal operator space structure,
we have $ \left\|\tilde{ \varphi }\right\|_{cb} =  \left\|
\varphi\right\|$. Thus, $ \left\| \varphi \right\|_{cb} \leq
\left\|\tilde{ \varphi }\right\|_{cb}= \left\|\tilde{ \varphi
}\right\|=\left\| \varphi\right\|$. This shows that $\varphi$ is
completely bounded. Similarly $\varphi^{-1}$ is also completely
bounded, so that $\varphi$ is a complete isomorphism. Thus,   $E$
and $F$ are completely (isometrically) isomorphic as operator
subspaces of $Max(X)$ and $Max(Y)$ respectively.   From this fact,
it follows that $\mu$-spaces can also be described as a submaximal
subspace of an injective commutative $C^*$- algebra,  because the
operator space structure is independent of the
 particular embedding.
\end{rem}
 Now we give a direct proof, different from \cite{tim04} of the fact that $\mu(X)$ is  the \em smallest \em
submaximal operator space structure on a given Banach space $X$.
\begin{theo}
Let $X$ be a Banach space. Then $\mu(X)$ is the smallest
submaximal space structure on $X$. \label{smallest}
\end{theo}
\begin{proof}
 Let $j$ be a complete isometric
embedding of $X$ in $Max(C(K)^{**})$ described in the definition
of $\mu$-spaces.   Let $\varphi: X \rightarrow Max(Y)$ be a
complete isometric embedding of $X$ into $Max(Y)$. Let the
sequence of matrix norms on $X$ obtained via this embedding be
denoted by $ \{  \left\| . \right\|^{Y} _{n} \}_{n \in \mathrm{N}}
$. Then for any $[x_{ij}] \in M_n
  (X)$, $$\left\| [x_{ij}]\right\|^{Y} _{n}=\left\| [\varphi(x_{ij})]\right\|
  _{Max(Y)} = \mbox{sup}\{  \left\|[u(\varphi(x_{ij})]) ) \right\| ;  u \in Ball(B(Y,
B(\mathcal{K})))\}$$ where the supremum is taken over all possible
maps $u : Y \rightarrow B(\mathcal{K})$ and over all Hilbert
spaces $\mathcal{K}$.\ Also, we have
 $$\left\| [x_{ij}]\right\|^{\mu} _{n}=\left\| [x_{ij}]\right\|
  _{Max(C(K)^{**})} = \mbox{sup}\{  \left\|[v(x_{ij})] \right\| ;  v \in Ball(B(C(K)^{**},
B(\mathcal{H})))\}$$ where the supremum is taken over all possible
maps $v : C(K)^{**} \rightarrow B(\mathcal{H})$ and over all
Hilbert spaces $\mathcal{H}$. Consider the following diagram.
\begin{equation}
\nonumber
\begin{CD}
X @>\varphi>>   Max(Y) @> u>> B(\mathcal{K})\\
@VV id V  \\
X @ > j>>      Max(C(K)^{**}) @> v>>B(\mathcal{H})
\end{CD}
\end{equation}
Since $\varphi^{-1}: \varphi(X) \subset Y \rightarrow X$ is
bounded, $w = j \circ id \circ \varphi^{-1}: \varphi(X)\subset
 Y \rightarrow C(K)^{**}$ is  bounded and $\left\| w \right\|= 1$. Since $C(K)^{**}$ is
 injective as a Banach space, $w$ has a bounded extension $\widetilde{w}: Y \rightarrow C(K)^{**}$
 with $\left\| \widetilde{w}\right\|=1$. Therefore, for  any map $v: C(K)^{**}\rightarrow B(\mathcal{H})$ with
$\left\| v\right\|\leq1$,  $\tilde{v}= v \circ \widetilde{w}: Y
\rightarrow B(\mathcal{H})$ is a bounded map and
 is a completely bounded map with
$\left\|\tilde{v}\right\|_{cb}\leq 1$ when regarded as a map from
$Max(Y)$ to $B(\mathcal{H})$.  Thus $\left\|
[x_{ij}]\right\|^{\mu} _{n}\leq \left\| [x_{ij}] \right\|^{Y}
_{n}$ for any $[x_{ij}] \in M_n (X)$. This shows that the
$\mu$-space structure on $X$ is the smallest submaximal space
structure on $X$.
\end{proof}
Maximal and minimal operator spaces are homogeneous, but in
general, submaximal spaces need not be homogeneous. Now we show
that $\mu$-spaces are homogeneous.
\begin{proposition}Every $\mu$-space is homogeneous.
\label{homo}
\end{proposition}
\begin{proof} Let $\varphi: \mu(X) \rightarrow \mu(X)$ be a bounded
linear map. Then $\varphi $  extends to a bounded linear map
$\tilde{\varphi}$  on $C(K)^{**}$, and it is then completely
bounded on $Max(C(K)^{**})$ and $\|\tilde{\varphi}\|_{cb}=
\|\tilde{\varphi}\|=\|\varphi\|$. But $\|\varphi\|_{cb}\leq
\|\tilde{\varphi}\|_{cb}= \|\varphi\|$. Hence $\mu(X)$ is
homogeneous.
\end{proof}
Completely bounded Banach- Mazur distance between two operator
spaces $X$ and $Y$ is defined as  $d_{cb}(X, Y) = \mbox{inf} \{
   \left\| \varphi \right\|_{cb}  \left\| \varphi^{-1} \right\|_{cb}
   : \varphi: X \rightarrow Y$  is a  complete
  isomorphism $\} $.
If $X$ is a Banach space, then the $\mu$-space structure on $X$
lies between the operator space structures $Min(X)$ and $Max(X)$.
  Now we shall show that  the cb distance between these spaces can
  be realized as the cb - norm of the identity mapping between
  them.
\begin{theo}For a Banach space $X$ ,we have: $$d_{cb}(Min(X), \mu(X))= \left\| id: Min(X)\rightarrow \mu(X)
  \right\|_{cb}$$ and $$d_{cb}(\mu(X), Max(X))= \left\| id: \mu(X)\rightarrow Max(X)
  \right\|_{cb}$$
  \end{theo}
  \begin{proof}  Let $T: \mu(X) \rightarrow Min(X) $ be a  complete
  isomorphism. Let $\tilde{T}$ denotes the same map regarded as a
  mapping from $\mu(X)$ to $\mu(X)$. Consider the following
  diagram.
\[
\xymatrix{
\mu(X)  \ar[rr]^T &&Min(X)\ar@{=}[d]\\
\mu(X) \ar[u]^{\tilde{T}}&&Min(X)\ar[llu]^{id}\ar[ll]^{T^{-1}}}
\]
  Here $id$ denotes the formal identity mapping regarded
  as a mapping from $Min(X)$ to $\mu(X) $.
  From the diagram, we get $$\left\| id: Min(X)\rightarrow \mu(X) \right\|_{cb}=\left\| \tilde{T}
   \circ T ^{-1} \right\|_{cb}\leq\left\| \tilde{T} \right\|_{cb}\left\| T^{-1} \right\|_{cb}
   .$$
   Since $\mu(X) $ is homogeneous (by above proposition~\ref{homo}), $\left\| \tilde{T} \right\|_{cb}= \left\| \tilde{T} \right\|=\left\| T \right\|
   = \left\| T \right\|_{cb}$, where the last equality is because
   of the minimal operator space structure of the range space of
   $T$. Thus we have $\left\| id: Min(X)\rightarrow \mu(X) \right\|_{cb}\leq \left\| T \right\|_{cb}
   \left\| T^{-1} \right\|_{cb}$.
   This shows that $d_{cb}(Min(X), \mu(X))= \left\| id: Min(X)\rightarrow \mu(X)
  \right\|_{cb}$. Similarly the other case follows.
  \end{proof}
 We show that among
submaximal spaces, the $\mu$-spaces are characterized by the
following universal property.
\begin{theo}\label{up}
A submaximal  space $X$ is a $\mu$-space up to  complete isometric
isomorphism if and only if for any submaximal space $Y$, any
bounded linear map $\varphi : Y \rightarrow X$ is completely
bounded with $\left\| \varphi \right\|_{cb} = \left\| \varphi
\right\|$.
\end{theo}
\begin{proof} Assume that $X = \mu(X)$. By definition of $\mu$-spaces,
 $ X = \mu(X) \subset Max(C(K)^{**})$, where $K = Ball(X^*)$.  Since $Y$ is submaximal, we have
$Y \subset Max(Z)$ for some
 operator space $Z$.  Now, $\varphi: Y \rightarrow
 \mu(X)$ can be regarded as a map $Y$ to $Max(C(K)^{**})$.  Since
 bidual of $C(K)$ is injective, there exists $\tilde{\varphi} : Z \rightarrow Max(C(K)^{**}) $
 with $\left\| \tilde{\varphi}\right\| = \left\| \varphi \right\|$ and
 $\tilde{\varphi}|_Y = \varphi$.  Considering $\tilde{\varphi} : Max(Z) \rightarrow
Max(C(K)^{**})$, we see that $\tilde{\varphi}$ is completely
 bounded and $\left\| \tilde{\varphi}
\right\|_{cb} = \left\| \tilde{\varphi} \right\|$. But $\left\|
\varphi \right\|_{cb} \leq \left\| \tilde{\varphi} \right\|_{cb} =
\left\| \tilde{\varphi} \right\| = \left\| \varphi \right\| $.
This  shows that $\left\| \varphi \right\|_{cb} =
\left\| \varphi \right\|$.\\
Conversely , take $Y = \mu(X)$ and $\varphi= id : \mu(X)
\rightarrow X $, the formal identity  mapping. Then by assumption,
$\left\| id \right\|_{cb} = \left\| id \right\|$ = 1.  Also, from
the above part, $\left\| id^{-1} \right\|_{cb} = \left\| id: X
\rightarrow \mu(X) \right\|_{cb}= \left\| id \right\| = 1 $.  Thus
$ id: X \rightarrow \mu(X)$ is a complete isometric isomorphism.
\end{proof}
\begin{rem}
We know that if $X$ has minimal operator space structure, then
every bounded linear map defined on another operator space with
values in $X$ is completely bounded. Also, we have shown that if
$X$ has the $\mu$-space structure, then any bounded linear map
from a submaximal space to $X$ is completely bounded. Now, let $X$
be endowed with any operator space structure $\{  \left\| .
\right\| _n \}_{n \in \mathrm{N}}$ such that $ \left\| [x_{ij}]
\right\| _n
 \leq   \left\| [x_{ij}]\right\| ^{\mu}_n $ for every $ [x_{ij}]\in M_n(X)$ , and for all $n \in \mathrm{N}$.
 In this case also, any bounded linear map $\varphi$ from a submaximal space
$ Y $ to $X$ is completely bounded with $\left\| \varphi
\right\|_{cb} = \left\| \varphi \right\|$. Here, the operator
space structure on $X$ is not submaximal, since the $\mu$-space
structure is the smallest submaximal structure on any normed
space. To see this, consider the following diagram.

\[
 \xymatrix{
Y  \ar[rr]^{\varphi}\ar[rrd]_{\tilde{\varphi}} &&X\\
&&\mu(X)\ar[u]_{id}}
\]
Since the identity mapping $id: \mu(X) \rightarrow X$ is a
complete contraction, using theorem \ref{up},  we have: $\left\|
\varphi \right\|_{cb} \leq \left\| id \right\|_{cb}\left\|
\tilde{\varphi} \right\|_{cb} \leq \left\| \tilde{\varphi}
\right\| = \left\| \varphi \right\| $. Thus, $\left\| \varphi
\right\|_{cb} = \left\| \varphi \right\|$.
\end{rem}
Now we make use of the universal property of $\mu$-spaces to show
that the class of $\mu$-spaces is stable under taking subspaces.
\begin{coro}
Let $Y \subset \mu(X)$. Then $Y$ is also a $\mu$-space.
\label{subsp}
\end{coro}
\begin{proof}  Let  $Z$ be any submaximal space. Then any bounded
linear map $\varphi: Z \rightarrow Y $ can be regarded as a map
from $Z$ to $\mu(X)$. By the universal property of $\mu$-spaces,
we see that $\left\| \varphi \right\|_{cb} = \left\| \varphi
\right\|$. This shows that $Y$ is a $\mu$-space.
\end{proof}
The following theorem gives a more general characterization of
$\mu$-spaces up to complete isomorphism.
\begin{theo}
A submaximal space $X$ is completely isomorphic to a $\mu$-space
if and only if for any submaximal space $Y$, any completely
bounded linear bijection $\varphi: X \rightarrow Y$ is a complete
isomorphism.
\end{theo}
\begin{proof}
Let $ \psi: X \rightarrow \mu(Z)$ be a complete isomorphism. Then
for any completely bounded linear bijection $\varphi: X
\rightarrow Y$, by theorem \ref{up}, $\psi \circ \varphi^{-1}: Y
\rightarrow \mu(Z)$ is completely bounded and  $\left\| \psi \circ
\varphi^{-1} \right\|_{cb} = \left\| \psi \circ \varphi^{-1}
\right\|$. Therefore, $\left\|  \varphi^{-1} \right\|_{cb}=
\left\| \psi^{-1}\circ \psi \circ \varphi^{-1} \right\|_{cb} \leq
\left\| \psi^{-1} \right\|_{cb} \left\| \psi \circ \varphi^{-1}
\right\|_{cb}\leq \infty$, showing that $\varphi$ is a complete
isomorphism. For the converse, take $Y$  as $ \mu(X)$   and
$\varphi$ as the formal identity map $id: X \rightarrow \mu(X) $.
\end{proof}
 Now we look at the case when the domain is endowed with the
$\mu$-space structure.
\begin{theo}
Let $X$ be an operator space. Then the formal identity map $id:
\mu(X) \rightarrow X$ is completely bounded if and only if for
every submaximal space $Y$, every bounded linear map $\varphi : Y
\rightarrow X$ is completely bounded. Moreover, we have: $ \|id:
\mu(X) \rightarrow X \|_{cb} = \mbox{sup} \{ \frac{\| u \|_{cb}}
{\|u\|}\}$ , where the supremum is taken over all bounded non zero
linear maps  $u: Y \rightarrow X $  and all submaximal spaces $Y$.
\end{theo}
\begin{proof} Assume that $id: \mu(X) \rightarrow X$ is completely
bounded with $\| id\|_{cb}= C$. Let $Y$ be a submaximal space and
$u: Y \rightarrow X $ be a bounded linear map. Let $\tilde{u}$
denotes the same map $u$ regarded as a map from $Y$ to $\mu(X)$.
Then by the universal property of $\mu$-spaces, $\tilde{u}$ is
completely bounded and $\left\| \tilde{u} \right\|_{cb} = \left\|
\tilde{u} \right\|= \| u \|$. Since $u = id \circ \tilde{u} $, we
have $\| u \|_{cb} \leq \|id\|_{cb}\| u \|_{cb}=C \|u\| < \infty$.
Thus $u$ is completely bounded. For the converse, take $Y $ as $
\mu(X)$, and $u$ as the identity map. Also, from the above
inequality, it follows that $ \|id: \mu(X) \rightarrow X \|_{cb} =
\mbox{sup} \{ \frac{\| u \|_{cb}} {\|u\|}\}$ , where the supremum
is taken over all bounded non zero linear maps  $u: Y \rightarrow
X $  and all submaximal spaces $Y$.
\end{proof}
The following theorem shows that an injective Banach space $X$ has
a unique submaximal space structure, or in other words $Min_S(X) =
Max_S(X)$, if $X$ is an injective Banach space.
\begin{theo}
If $X$ is an injective Banach space, then $\mu(X)$ is completely
isometrically isomorphic to $Max(X)$.
\end{theo}
\begin{proof}
Consider the formal identity map $id: \mu(X) \rightarrow Max(X)$.
By definition of $\mu$-spaces, $\mu(X)\subset Max(C(K)^{**})$ and
since $X$ is injective as a Banach space, $id$ extends to a
bounded linear map $\widetilde{id}:Max(C(K)^{**}) \rightarrow
Max(X) $ with $\| \widetilde{id} \| = \| id \|=1$. Since domain
has maximal operator space structure, we have $\| \widetilde{id}
\|_{cb} =1$ and hence $\| id \|_{cb} =1$. Also, $\| id^{-1}
\|_{cb} =1$, showing that $id: \mu(X) \rightarrow Max(X)$ is a
complete isometric isomorphism.
\end{proof}
We know that the converse of the above theorem is not true. For
example, the space $\ell_1^2$ has a unique operator space
structure \cite{gp03}, but it is not an injective Banach space.
The following theorem describes some equivalent conditions for the
uniqueness of the submaximal space structures.
\begin{theo}
For a Banach space $X$, the following are equivalent. \\(1) $X$
has a unique submaximal space structure. \\(2) $\mu(X)= Max(X)$
completely isometrically.\\(3) Any bounded linear map $\varphi: X
\to B(\mathcal H)$ admits a bounded extension
$\widetilde{\varphi}: C(Ball(X^*)) \to B(\mathcal H)$ with
$\|\widetilde{\varphi}\|= \| \varphi\|$.
\end{theo}
\begin{proof}
It is clear from the definition that (1) $\Leftrightarrow $(2).
Now to prove (2) $\Rightarrow $(3), regard $\varphi$ as a map from
$ \mu(X) \to B(\mathcal H)$, we see that $\varphi$ is completely
bounded and $\|\varphi\|_{cb}= \|\varphi\|$. Since
$\mu(X)\hookrightarrow Max(C(Ball(X^*)))$ completely isometrically
and since $B(\mathcal H)$ is injective, there exists an extension
$\widetilde{\varphi}: Max(C(Ball(X^*))) \to B(\mathcal H)$ with
$\|\widetilde{\varphi}\|_{cb}= \| \varphi\|_{cb}$. Thus
$\widetilde{\varphi}: C(Ball(X^*)) \to B(\mathcal H)$ satisfies
$\|\widetilde{\varphi}\|= \|\widetilde{\varphi}\|_{cb}=
\|\varphi\|_{cb}= \| \varphi\|$.\\
 Assume that any bounded
linear map $\varphi: X \to B(\mathcal H)$ admits a bounded
extension $\widetilde{\varphi}: C(Ball(X^*)) \to B(\mathcal H)$
with $\|\widetilde{\varphi}\|= \| \varphi\|$.\\
Clearly  $\left\| [x_{ij}]\right\|^{\mu} _{n}\leq \left\| [x_{ij}]
\right\|^{max} _{n}$ for any $[x_{ij}] \in M_n (X)$. By definition
of maximal spaces,
$$\left\| [x_{ij}]\right\|^{max} _{n}=\mbox{sup}\{  \left\|[\varphi(x_{ij})] ) \right\| ;  \varphi \in Ball(B(X,
B(\mathcal{H})))\}$$ where the supremum is taken over all possible
bounded linear maps $\varphi : X \rightarrow B(\mathcal{H})$ and
over all Hilbert spaces $\mathcal{H}$. Also,
\begin{align*}
 \left\| [x_{ij}]\right\|^{\mu} _{n} = & \left\| [x_{ij}]\right\|
  _{Max(C(Ball(X^*)))} \\
  = & \mbox{sup}\{  \left\|[v(x_{ij})] \right\| ;  v \in Ball(B(C(Ball(X^*)),
B(\mathcal{H})))\}
 \end{align*}
 where the supremum is taken over
all possible bounded linear maps \\$v : C(Ball(X^*))\rightarrow
B(\mathcal{H})$ and over all Hilbert spaces $\mathcal{H}$. By the
assumed extension property of $X$, corresponding to any $u \in
Ball(B(X, B(\mathcal{H})))$, we have an extended function
$\tilde{u}\in Ball(B(C(Ball(X^*)), B(\mathcal{H}))) $, so that
 $\left\| [x_{ij}]\right\|^{\mu} _{n}\geq \left\| [x_{ij}]
\right\|^{max} _{n}$ for any $[x_{ij}] \in M_n (X)$.\\
Thus $\mu(X)= Max(X)$, showing that (3) $ \Rightarrow $ (2).
\end{proof}
\begin{rem}
Since every injective Banach space has a unique submaximal space
structure, every injective Banach space $X$ has the above
described extension  property.
\end{rem}
 A  $Q$-space is an (operator) quotient of a minimal space
\cite{er}.  Note that if $X$ is a $Q$-space, then $X^*$ is a
submaximal space.  Conversely , the dual of a submaximal space is
a $Q$-space. Eric Ricard \cite{ricard} introduced the maximal $Q
$-space structure on a Banach space $X$ denoted by  $Max_Q(X)$,
where the matrix norms are defined as:
\begin{center}
$\left\| [x_{ij}] \right\|=\mbox{sup} \{ \left\|
[u(x_{ij})]\right\|_{M_n(E)}; u: X \rightarrow E $,   $E$ a $Q
$-space and $\left\| u \right\| \leq 1$ $\}$
\end{center}

 We now prove the
duality relations between the $\mu(X)$ and $Max_Q(X)$.
\begin{theo}
For any Banach space $X$, we have: $(Max_Q(X))^* = \mu(X^*)$ and
$(\mu(X))^* = Max_Q(X^*)$.
\end{theo}
\begin{proof} Note that $(Max_Q(X))^* $ is a submaximal space
structure on $X^*$, so that by theorem~\ref{smallest}, the formal
identity map $id:(Max_Q(X))^* \rightarrow \mu(X^*) $ is a complete
contraction. Also, the embedding  $X \subset (\mu(X^*))^*$ gives a
$Q$-space structure on $X$. Hence the identity map $id:Max_Q(X)
\rightarrow (\mu(X^*))^* $ is a complete contraction. Taking the
duals, we see that $id:\mu(X^*) \rightarrow (Max_Q(X))^* $ is a
complete contraction. Thus $(Max_Q(X))^* = \mu(X^*)$. The other
part follows by duality.
\end{proof}
\begin{coro}
An operator space $X$ is a $\mu$-space if and only if its bidual
$X^{**}$ is a $\mu$-space.
\end{coro}
\begin{proof} Let $X = \mu(X)$. Then $X^* = \mu(X)^*= Max_Q(X^*)$,
so that   $X^{**} = (Max_Q(X^*))^* = \mu(X^{**})$. Thus $X^{**}$
is a $\mu$-space. The converse part follows from the fact that $X
\subset X^{**}$ and from the Corollary \ref{subsp}
\end{proof}

 \em Acknowledgements \em. The first
author acknowledge the financial support by University Grants
Commission of India, under Faculty Development Programme. The
authors are grateful to Prof. Gilles Pisier and Prof. Eric Ricard
for having many discussions on this subject. Also, the first
author would like to thank Prof. V.S. Sunder, the organizer of the
Instructional Workshop on the Functional Analysis of Quantum
Information Theory in the Institute of Mathematical Sciences,
Chennai, India, in January 2012, which gave an opportunity to
interact with Prof. Gilles Pisier.

Corresponding Author : Vinod Kumar. P

\end{document}